\documentclass[letterpaper, 10 pt, conference]{IEEEconf}

\IEEEoverridecommandlockouts
\usepackage{amssymb}

\usepackage{graphicx,color}
\graphicspath{{Images/}}
\usepackage{amsmath}
\usepackage{subfig}
\newtheorem{theorem}{Theorem}
\newtheorem{proposition}[theorem]{Proposition}
\newtheorem{corollary}[theorem]{Corollary}
\newtheorem{lemma}[theorem]{Lemma}

\newtheorem{definition}[theorem]{Definition}

\newtheorem{assumption}[theorem]{Assumption}

\def\ba{\begin{array}}
\def\ea{\end{array}}
\def\bi{\begin{itemize}}
\def\ei{\end{itemize}}
\def\mR{\mathbb{R}}

\def\mZ{\mathbb{Z}}

\def\mE{\mathbb{E}}
\def\m1{1}
\def\eps{\varepsilon}

\def\cB{\mathcal{B}}
\def\cA{\mathcal{A}}
\def\b1{{\mathbf 1}}

\def\medno{}
\def\diag{\operatorname{diag}}
\DeclareMathOperator\supp{supp}

\begin{document}

\title{Computable convergence rate bound for ratio consensus algorithms}

\author{Bal\'azs Gerencs\'er%
  \thanks{B. Gerencs\'er is with the Alfr\'ed R\'enyi Institute of Mathematics, Budapest, Hungary and E\"otv\"os Lor\'and University, Department of Probability and Statistics, Budapest, Hungary, {\tt\small gerencser.balazs@renyi.hu} He was supported by NRDI (National Research, Development and Innovation Office) grant FK 135711 and KKP 137490, the J\'anos Bolyai Research Scholarship of the Hungarian Academy of Sciences and by the ÚNKP-20-5 New National Excellence Program of the Ministry for Innovation and Technology from the source of the NRDI Fund.}
}

\maketitle

  \begin{abstract}
    The objective of the paper is to establish a computable upper bound for the almost sure convergence rate for 
    a class of ratio consensus algorithms defined via column-stochastic matrices.
    Our result extends the works of Iutzeler et al.\ \cite{iutzeler2013analysis} on similar bounds that have been obtained in a more restrictive setup with limited conclusions. The present paper complements the results of Gerencsér and Gerencsér \cite{gerencsr2019tight}, identifying the exact almost sure convergence rate of a wide class of ratio consensus algorithms in terms of a spectral gap, which is, however, not computable in general. The upper bound provided in the paper will be compared to the actual rate of almost sure convergence experimentally on a range of modulated random geographic graphs with random local interactions. 
\end{abstract}

%


\section{Introduction}

\medno
\emph{Ratio consensus} algorithms were initially proposed in a special form by Kempe et al.\ \cite{kempe2003gossip} under the name \emph{push-sum}, with its scope being extended later in \cite{benezit2010weighted} under the name \emph{weighted gossip}. The basic setup of these algorithms is a directed graph or network with values associated to each node. The objective is the design of a communication protocol for the computation of the average of the initial input values given at the nodes, using only local, directed, possibly asynchronous communication. Ratio consensus algorithms became the building blocks of further methods requiring distributed computation, such as the analysis of  sensors networks \cite{giridhar2006toward}, the spectral analysis of a network \cite{kempe2008decentralized} or distributed optimization \cite{nedic2014distributed}, just to highlight a few.

\medno
For the sake of historical context note that ratio consensus is an extension of classic \emph{gossip algorithms} for average consensus, see \cite{tsitsiklis:phd1984}, \cite{blondel2005convergence}, in which the graph is not directed, and updating the values is restricted to a randomly chosen communicating pair of nodes, replacing their values by the average. Gossip algorithms are linear: updates are defined via (left-)multiplication by a \emph{doubly stochastic} random matrix. 

\medno
As soon as real-life communication conditions are included in the analysis, additional care is needed as is the case for packet loss \cite{frasca2013large} where the large size of the network allows controlling the error, or in case of delay \cite{olfati2004consensus}, where this delay needs to be bounded by spectral properties to achieve average consensus.

\medno
The exponential rate of convergence in mean square sense for gossip algorithms, with i.i.d. selection of communicating pairs, has been determined in \cite{boyd2006randomized}.
A significant advance, assuming strictly stationary edge selection was presented in \cite{picci2013almost} establishing almost sure (a.s.) exponential rate of convergence via a spectral gap in the context of Oseledec's multiplicative ergodic theorem. 
There is a vast literature for in-depth understanding of such algorithms, for a wider perspective including, e.g., distributed optimization and further references we refer to the survey \cite{MAL-051}.

\medno
Ratio consensus algorithms were designed for possibly asynchronous communication protocols on a directed graph, leading to updates defined via multiplication by a \emph{column stochastic} random matrix, which in itself would fail to reach average consensus. This shortcoming is compensated by running an additional process, allocating weights to each node, with initial weights equal to $1,$ and considering the quotients value/weight, which are then expected to converge to the required average value for all nodes.  

\medno
Almost sure convergence of ratio consensus algorithms has been established under a variety of settings
see \cite{kempe2003gossip} or \cite{benezit2010weighted}, even for the case of communication protocols with bounded communication delays \cite{hadjicostis2014average}. However the question on the exact rate of a.s.\ convergence, raised back in 2010, see  \cite{benezit2010weighted}, was open for a decade. 

\medno
A partial answer to the question on the a.s.\ rate of a ratio consensus algorithm was given in Iutzeler et al.\ \cite{iutzeler2013analysis},  providing an upper bound along an unspecified, infinite subset of the timeline. More recently, the paper of Gerencs\'er and Gerencs\'er \cite{gerencsr2019tight} identified the \emph{exact rate of a.s.\ convergence} as the spectral gap, in the context of Oseledec's multiplicative ergodic theorem, of the associated random matrix process under very general conditions. However, the spectral gap is known to be \emph{uncomputable} in general
\cite{tsitsiklis1997lyapunov}.

\medno
The purpose of this paper is to provide a \emph{computable} upper bound for the rate of a.s.\ convergence along the full timeline, under technical assumptions that are weaker than those of \cite{iutzeler2013analysis}. This result is obtained by combining arguments of \cite{iutzeler2013analysis}, which we simplify and extend with the results of \cite{gerencsr2019tight}. 
Apart from a technical tool borrowed from \cite{gerencsr2019tight} we provide a transparent and self-contained proof.

\section{Technical setup and main result}
\label{sec:prelim_0}

\medno
To describe the technical details in terms of algebraic operations let $p$ be the number of agents, or equivalently, the number of nodes of the communication graph. Let $x_0 \in \mathbf R^p$ be a column vector composed of the initial values associated with the nodes in some prefixed order. Our objective is to compute the average $\bar x = \sum_{i=1}^p x_0^i /p.$ Let $w_0=\b1 \in \mathbf R^p$ an auxiliary vector, the components of which are called weights. At any time $n \ge 1$ the transmission of an (identical)  fraction of values and weights results in updated value and weight vectors as follows:
\begin{equation}
\label{eq:Update_x_w}
x_{n}=A_{n}x_{n-1},\qquad
w_{n}=A_{n}w_{n-1}, 
\end{equation}
where $(A_n), n\ge 1$ is an i.i.d.\ sequence of non-negative column-stochastic matrices, implicitly representing all constraints imposed by the network and specifying the local, possibly asynchronous transmissions without packet loss.   

\medno
The average at agent $i$ at time $n$ is then estimated by the readout $x^i_n/w^i_n$. A simple interpretation of the algorithm is obtained by thinking of $x^i_n/w^i_n$, as a concentration of a substance in some solvent, properly re-scaled. Note that we can write $\bar x = \b1^\top x_0/p = \b1^\top x_0/(\b1^\top w_0)$. Observe also that 
$$
\b1^\top x_n=\b1^\top A_n A_{n-1}\cdots A_1 x_0=\b1^\top x_0
$$
since the matrices $A_k$ are column-stochastic. Thus the overall average of $x_n$ is conserved, similarly for $w_n,$ thus $\b1^\top w_n = p$ for all $n.$ The rate of a.s.\ convergence is defined as 
\begin{equation}
\limsup_{n\to\infty} \frac 1 n \log  \sum_{i=1}^p\bigg|\frac{x_n^i}{w_n^i} -\bar x\bigg|.
\end{equation}
A significant advance over previous works was the determination of a theoretical and tight upper bound for the rate of a.s. convergence under a variety of reasonable conditions, see Theorems 12-19 of \cite{gerencsr2019tight}, thus settling an open problem raised back in 2010, see \cite{benezit2010weighted}. 

\medno
In order to clarify the main result to be stated let us revisit the linear algebraic arguments of \cite{iutzeler2013analysis} in preparation for the analysis of $x^i_n/w^i_n - \bar x.$
Let $M_n = A_nA_{n-1}\cdots A_1$ denote the total effect of the updates on $x_0$ or $w_0$ until time $n$. Let $I$ denote a $p \times p$ identity matrix  and let $J=\b1\b1^\top/p$. Note that $x_n = M_n x_0$ can be decomposed as
\begin{align*}
  x_n &= M_nJ x_0 + M_n(I-J)x_0 = M_n\b1 \bar x + M_n(I-J)x_0\\
      &= w_n\bar x + M_n(I-J)x_0.
\end{align*}
Thus, at agent $i,$ the ratio consensus algorithm will yield
\begin{equation}
  \label{eq:x_i_per_w_i_via_M_N}
  \frac{x_n^i}{w_n^i} 
  = \bar x + \frac{e_i^{\top}M_n(I-J)x_0}{w_n^i},
\end{equation}
where $e_i$ is the unit vector with a single 1 at position $i$.

\medno
It follows that the error of ${x_n^i}/{w_n^i}$ is largely controlled by behavior of the  matrix $N_n := M_n(I-J)$. We can get a useful alternative expression by noting that $A_m$ being column-stochastic implies 
\begin{equation}
\label{eq:i-j}
(I-J)A_m(I-J)=A_m(I-J).  
\end{equation}
Applying this repeatedly for $N_n$ we get the expression
\begin{equation}
\label{eq:Express_N__n_as_Product}
N_n=(A_n(I-J)) \cdot (A_{n-1}(I-J))\cdots (A_1(I-J)).
\end{equation}

\begin{theorem}
  \label{thm:Main}
  Let Assumption \ref{assumption:typicalsetup} given below be satisfied. Then with $\eta_2 = \log\rho\left(\mE(A_{1}^{\otimes 2})(I-J)^{\otimes 2}\right)$ we have a.s.  
  \begin{align*}
    \limsup_{n\to\infty} \frac 1 n \log  \sum_{i=1}^p\bigg|\frac{x_n^i}{w_n^i} -\bar x\bigg| \le {\frac {\eta_2} 2} < 0.
  \end{align*}
\end{theorem}

\medno
The technical results of the paper also complement previous results on the rate of convergence of linear gossip algorithms, defined via \emph{doubly-stochastic} matrices, which were available both in mean-squared and a.s. sense, as in \cite{boyd2006randomized}, \cite{picci2013almost}. For the current statement see Corollary \ref{cor:lin_moment}.

\medno
The enhance readability the necessary technical results of \cite{gerencsr2019tight} are restated under the specific conditions of the present paper collected in Assumption \ref{assumption:typicalsetup}.

\section{Sequential primitivity}
\label{sec:prelim}

In what follows we present the basic technical assumptions needed for the application of the results of \cite{gerencsr2019tight}. It is intuitively clear that in order to get convergence to the average at all nodes, we need to ensure all-to-all influence. Technically speaking, we should require that the matrix product $A_n\cdots A_1$ is positive for large enough, possibly random $n$. This leads to the following definition, in a more general and deterministic context, formulated in \cite{protasov2012sets} as follows:

\begin{definition}
  A set $\cA$ of $p\times p$ non-negative matrices is \emph{primitive} if a strictly positive product can be formed by some elements of it, repetitions allowed.
\end{definition}

\medno
Recall that a non-negative matrix is called \emph{allowable} if all rows and all columns contain at least one strictly positive element, see \cite{seneta2006non}. Now, if the matrices $A_n \in \cA$ are chosen according to some random process, we get a natural extension of the notion of primitivity:

\begin{definition}
  A strictly stationary process $(A_n), n\ge 1$ of $p\times p$ non-negative allowable matrices is \emph{sequentially primitive} if $A_\tau A_{\tau-1}\cdots A_1$ is strictly positive for a finite stopping time $\tau$.
\end{definition}

\medno
Sequential primitivity is easily established for an i.i.d.\ sequence of matrices
by the lemma below:
\begin{lemma}
\label{lem:SEQ_PRIM_IID}
  Consider a set of $p\times p$ non-negative matrices $\cal A$, such that all $A\in \cal A$ is allowable. Assume that $\cal A$ is primitive. Let $\mu$ be a fully supported distribution on $\cal A$, i.e., $\supp \mu = \cal A$. Consider the i.i.d.\ sequence $(A_n), n\ge 1$, distributed according to $\mu$. Then $(A_n)$ is sequentially primitive.
\end{lemma}

\medno
The proof will be given in the Appendix. For the sake of historical perspective we note that the following simple alternative sufficient condition for sequential primitivity was given in \cite{benezit2010weighted}:
\begin{proposition}
  \label{lm:benezit_prim}
  Let $(A_n), n\ge 1$ be an i.i.d.\ sequence of $p\times p$ matrices. Assume that $A_1$ has a strictly positive diagonal almost surely, and $\mE A_1$ is irreducible.  Then $(A_n)$ is sequentially primitive.
\end{proposition}

\medno
We note in passing that in  \cite{iutzeler2013analysis} the conditions of \cite{benezit2010weighted} are assumed to be satisfied for $\cal A,$ and in addition $|\cA|<\infty$ is assumed. It is easily seen that the condition that $\mE A_1$ is irreducible is in fact necessary for sequential primitivity: 

\begin{lemma}
	\label{lem:EA_1_Irreducible}
	Let $(A_n), n\ge 1$ be a strictly stationary sequence of non-negative matrices. Assume that $(A_n)$ is sequentially primitive. Then  $\mE A_{1}$ is irreducible.
\end{lemma}

\medno
The proof will be given in the Appendix.
Note, however, that irreducibility in expectation is not sufficient by itself. For example consider an i.i.d.\ sequence distributed uniformly on
$\big(\begin{smallmatrix}
1 & 0\\
0 & 1
\end{smallmatrix}\big)$
and
$\big(\begin{smallmatrix}
0 & 1\\
1 & 0
\end{smallmatrix}\big)$. On the other hand, the condition that $A_1$ has a strictly positive diagonal almost surely is not necessary for sequentially primitivity. For example consider the i.i.d. sequence of matrices with distributed uniformly on
$\big(\begin{smallmatrix}
0 & 1\\
1 & 0
\end{smallmatrix}\big)$
and
$\big(\begin{smallmatrix}
1 & 1\\
0 & 1
\end{smallmatrix}\big),$
which is readily seen to be sequentially primitive.

\begin{lemma}
  \label{lem:tensor_seq_prim}
  Let $(A_n), n\ge 1$ be a strictly stationary sequence of non-negative matrices. Assume that $(A_n)$ is sequentially primitive. Then for any $k\in\mZ^+$, $(A_{n}^{\otimes k}), n\ge 1$ is also sequentially primitive.
\end{lemma}

\begin{proof}
  Observe that for any $n$
  $$
  A_n^{\otimes k}A_{n-1}^{\otimes k}\cdots A_1^{\otimes k} = \left(A_n A_{n-1}\cdots A_1\right)^{\otimes k},
  $$
  hence the left hand side is strictly positive exactly if $A_n A_{n-1}\cdots A_1$ is strictly positive, proving the claim. 
\end{proof}

\medno
The examples above indicate that sequential primitivity may be a fundamental concept, and this is indeed fully justified in \cite{gerencsr2019tight}, see in particular Theorem 19, restated as Proposition \ref{thm:ei Mnx vs ei Mnw COLUMN STOCH} below. The conditions of the latter, to be used throughout the paper, can be reformulated as follows:

\begin{assumption}
	\label{assumption:typicalsetup}
	Let $\cA$ be a set of  $p \times p$ matrices, and let $(A_n), n\ge 1$ be a $\cA$-valued stochastic process, satisfying the following conditions:  
	\begin{itemize}
		\item $\cA$ is a Borel set of non-negative, allowable, column-stochastic matrices. 
		\item $\cA$ is primitive.
		\item $(A_n), n\ge 1$ is an i.i.d.\ sequence of matrices in $\cA$.
		\item The distribution of $A_1$ is fully supported on $\cA$.
		\item Setting $\alpha_n:=\min_{i,j}\{A_{n}^{i,j}:A_{n}^{i,j}\neq 0\}$, we have $\mE\log^-\alpha_{1}>-\infty$.
	\end{itemize}
\end{assumption}

\medno
It is readily seen that the above assumptions on $(A_n), n\ge 1$ are significantly weaker than those in \cite{iutzeler2013analysis}, requiring the conditions of \cite{benezit2010weighted} to be satisfied for a finite $\cA$.

\section{Tight bounds for a.s. convergence}
\label{sec:tightbounds}

\medno
In this section we highlight the relevant conditions and results of \cite{gerencsr2019tight}, specialized to the context of the present paper. First of all we note that the condition $\mE \log^+ \Vert A_1 \Vert < \infty$ required by the F\"urstenberg--Kesten theorem and also by Oseledec's theorem, restated as Proposition 1 and 2 in \cite{gerencsr2019tight}, and used throughout that paper, is automatically satisfied for column-stochastic matrices. Following these fundamental results, let $\lambda_1$ and $\lambda_2$ be the first and second largest Lyapunov exponents associated with $(A_n).$
 
\medno
The condition of Theorem 8 of \cite{gerencsr2019tight}, serving as a benchmark for subsequent discussion, requiring that $A_n$ is non-negative and allowable for all $n,$ and that the process $(A_n)$ is sequentially primitive, is implied by Assumption \ref{assumption:typicalsetup}. 
Condition 11 of \cite{gerencsr2019tight}, imposing a kind of lower bound on the strictly positive elements of $A_n,$ is required by Assumption  \ref{assumption:typicalsetup} in identical form. 

Finally, the condition $ \lambda_1 - \lambda_2 > 0,$ required in the first main result of \cite{gerencsr2019tight}, stated as Theorem 12, is in fact implied by Assumption \ref{assumption:typicalsetup}, see Theorem 36 of \cite{gerencsr2019tight}. 

\medno 
Now we are in a position to restate Theorem 19 of \cite{gerencsr2019tight} in a specialized form, the reference result for identifying the convergence rate of ratio consensus,
with $w= \b1$ as follows, with $e_i, \, i=1, \dots ,p$ denoting the $i$-th unit vector: 

\begin{proposition}
	\label{thm:ei Mnx vs ei Mnw COLUMN STOCH}
	Let Assumption \ref{assumption:typicalsetup} be satisfied. Then for an arbitrary vector of initial values $x \in \mathbb R^p$ and initial weights $w = \b1,$ we have for all $i=1, \ldots, p$
	\begin{equation*}
	\limsup_{n \rightarrow \infty} {\frac 1 n} \log \left \vert \frac{e_i^\top M_n x}{e_i^\top M_n \b1}  - \bar x  \right \vert \le \lambda_2 < 0 \quad {\rm a.s.}
	\end{equation*}
\end{proposition}

\medno
In the current development will need a critical auxiliary technical result on the evolution of the weight vector $w_n=A_nA_{n-1}\cdots A_1\b1$.
\begin{lemma}
  \label{lem:w_subexp1}
  Let Assumption \ref{assumption:typicalsetup} be satisfied. Then 
  $1/\min_iw_{n}^i$ is sub-exponential: 
  $$\limsup_{n\to\infty}\frac 1 n \log \frac{1}{\min_iw_{n}^i} \le 0.$$
\end{lemma}
The proof will be given in the Appendix. 

\section{Restricted contraction of $A_nA_{n-1}\cdots A_1$}
\label{sec:contract}

\medno
We will estimate higher order moments of $N_n$ by considering higher order tensor products $N_n^{\otimes k}$ with some $k\in\mZ^+$. To set the notations, let $B_n = A_n(I-J)$ and $\|\cdot\|_F$ denote the Frobenius norm of a square matrix.
Note that for any square matrix $S$, the sum of squares of the elements of $S,$ expressing $\|S\|^2_F,$ is in fact a sum of selected elements of $S\otimes S$, and thus we can write, with an appropriate linear functional $L$, 
\begin{equation}
\label{eq:frob_vs_lin}
\left\|S\right\|^2_F = L (S \otimes S).
\end{equation}
Finally, let $\rho(\cdot)$ denote the spectral radius a square matrix.

\begin{lemma}
  \label{lem:frob_of_N_n_rate}
  Under Assumption \ref{assumption:typicalsetup} we have with $\eta_{2k} =\log\rho\left(\mE(B_{1}^{\otimes 2k})\right)$
  \begin{align}
    \label{eq:frob_of_N_n_rate}
    \limsup_{n\to\infty} \frac 1 n \log \mE\left\|N_n^{\otimes k}\right\|^2_F &\le \eta_{2k} < 0.
  \end{align}
\end{lemma}
The lemma above was given in \cite{iutzeler2013analysis} for the case $k = 1$ with a proof, relying on another paper of the authors. Lemma \ref{lem:frob_of_N_n_rate} is thus a generalization for all integers $k$, together with a direct, simple proof. This generalization is also relevant in estimating higher order moments of the error obtained in the course of linear gossip algorithms.

\begin{proof}
Taking the $2k$-th tensor power of \eqref{eq:Express_N__n_as_Product}, followed by taking expectation, recalling that $(A_m)$ is i.i.d., we get
\begin{align*}
  \mE(N_n^{\otimes 2k})=\mE\left(B_{n}^{\otimes 2k}\right) \cdots \mE\left(B_{1}^{\otimes 2k}\right)=\left(\mE(B_{1}^{\otimes 2k})\right)^n.
\end{align*}
From here using \eqref{eq:frob_vs_lin} we get
\begin{equation*}
  \mE\left\|N_n^{\otimes k}\right\|^2_F = \mE \left(L N_n^{\otimes 2k}\right) = L \mE \left(N_n^{\otimes 2k}\right) = L \left(\mE(B_{1}^{\otimes 2k})\right)^n.
\end{equation*}
$L$ is a fixed linear functional, thus
\begin{align*}
  \limsup_{n\to\infty} &\frac 1 n \log \mE\left\|N_n^{\otimes k}\right\|^2_F \le\\
  &\le \limsup_{n\to\infty} \frac 1 n \log \left(\|L\|\! \cdot\! \left\| \mE(B_{1}^{\otimes 2k})^n \right\|\right) = \eta_{2k},
\end{align*}
using a standard expression of the spectral radius. This confirms the first inequality in \eqref{eq:frob_of_N_n_rate}.

For the second part of the inequality, note that the expectation of the column-stochastic $A_1^{\otimes 2k}$ is itself column-stochastic and the primitivity assumptions provide irreducibility by Lemmas \ref{lem:SEQ_PRIM_IID} and \ref{lem:EA_1_Irreducible}. Therefore the Perron-Frobenius theorem ensures a single maximal eigenvalue with left eigenvector $\b1^{\otimes 2k \top}$. Therefore multiplying $\mE(A_1^{\otimes 2k})$ by the projection $P_1=(I^{\otimes 2k}-J^{\otimes 2k})$, which maps $\mR^{\otimes 2k}$ into the orthogonal complement of $\b1^{\otimes 2k \top}$ and acts as identity there will result in the stable matrix $\tilde B = \mE(A_1^{\otimes 2k})P_1$. By the same observation, $P_1 (I-J)^{\otimes 2k} = (I-J)^{\otimes 2k}$. Consequently we may express the log spectral radius of interest as
\begin{equation}
  \label{eq:exp_tensor_itself}
\limsup_{n\to\infty} \frac 1 n\log \left\|\left(\mE(A_1^{\otimes 2k})P_1 (I-J)^{\otimes 2k}\right)^n\right\|  
\end{equation}
Note that \eqref{eq:i-j} can be extended to the tensor power, also inserting $P_1$ using the invariance observed above, i.e.,
$$
(I-J)^{\otimes 2k} \tilde B (I-J)^{\otimes 2k} = \tilde B (I-J)^{\otimes 2k}.
$$
Repeatedly applying this to the product inside the expression of \eqref{eq:exp_tensor_itself} we arrive at
$$
\limsup_{n\to\infty} \frac 1 n\log \left\|\tilde B^n (I-J)^{\otimes 2k}\right\|\le \log \rho(\tilde B) < 0.
$$
\end{proof}

\begin{corollary}
  \label{cor:tens_as}
  Under Assumption \ref{assumption:typicalsetup},
  \begin{equation}
    \label{eq:frob_as}
    \limsup_{n\to\infty} \frac 1 n \log \left\|N_n^{\otimes k}\right\|^2_F \le \eta_{2k}\quad\rm{a.s.}.    
  \end{equation}
\end{corollary}
\begin{proof}
  Given the moment bound of Theorem \ref{lem:frob_of_N_n_rate}, by a standard combination of the Chernoff-inequality and the Borel-Cantelli lemma, for any fixed $l\in\mZ^+$ the event $\frac 1 n \log \left\|N_n^{\otimes k}\right\|^2_F > \eta_{2k}+\frac 1 l$ occurs finitely many times a.s. which then combined for all $l\in\mZ^+$ confirms the claim.
\end{proof}

\begin{proof}[Proof of Theorem \ref{thm:Main}]
  We perform a slight rearrangement so that we can introduce the $2k$-th power of a single term. For any positive $p$-tuple of $u_i$ we may write
  $
  \log\sum_{i=1}^p u_i \le \log (p \max_i u_i) = \frac 1 {2k} \log (p \max_i u_i)^{2k}.
  $
  For our target expression this translates to 
  \begin{equation}
  \label{eq:pf_thm_rear}
  \begin{aligned}
    \limsup_{n\to\infty} &\frac 1 n \log \sum_{i=1}^p \bigg|\frac{x_n^i}{w_n^i} -\bar x\bigg|\le\\
                         &\limsup_{n\to\infty} \frac 1 {2kn} \log \max_i\left|\frac{e_i^{\top}N_nx_0}{w_n^i}\right|^{2k}.
  \end{aligned}
  \end{equation}
  To get a hand on this quantity, recall that the denominator is sub-exponential by Lemma \ref{lem:w_subexp1}, thus it does not alter the rate. For the numerator, there holds for some $c_{2k}>0$
  \begin{equation}
    \label{eq:pf_thm_bound}
    \begin{aligned}
    |e_i^{\top}N_n x_0|^{2k} &\le c_{2k} \|e_i^{\top}\|^{2k} \|N_n\|_F^{2k} \|x_0\|^{2k}\\
    &= c_{2k}\|e_i^{\top}\|^{2k} \|N_n^{\otimes k}\|_F^{2} \|x_0\|^{2k},
  \end{aligned}
  \end{equation}
  using that for any $S$, both $\|S^{\otimes k}\|_F^{2},\|S\|_F^{2k}$ express the sum of all $k$-fold products of the squared elements of $S$ and are thus equal.
  Plugging this back to \eqref{eq:pf_thm_rear} and using the result of Corollary \ref{cor:tens_as} we get the upper bound of $\eta_{2k}/{2k}$ on the rate. Set $k=1$ to confirm the claim.
\end{proof}
Combining Lemma \ref{lem:frob_of_N_n_rate} with \eqref{eq:pf_thm_bound} above we get a $2k$-moment convergence rate bound for \emph{linear} consensus:
\begin{corollary}
  \label{cor:lin_moment}
  Under Assumption \ref{assumption:typicalsetup} further requiring $A_n$ to be doubly stochastic there holds
  \begin{equation*}
    \limsup_{n\to\infty} \frac 1 n \log \mE \|M_nx_0 - \bar x\b1\|^{2k} \le \eta_{2k}.
  \end{equation*}
\end{corollary}

\section{Optimizing the tensor exponent $k$}
\label{sec:k_optim}

As we have seen in the proof of Theorem \ref{thm:Main}, the main term in \eqref{eq:pf_thm_bound} quantifying the error becomes $\|N_n\|_F^{2k}$ once $k$ is chosen, which is then further bounded with the tools obtained before.
Directly examining $\|N_n\|_F^{2k}$ we would get the so-called $s=2k$-th mean Lyapunov exponent that could be defined for any $s>0$, see \cite{arnold1986lyapunov}, as 
\begin{equation}
  \lambda^{s} = \lim_{n \rightarrow \infty} {\frac 1 n} \, \mE \Vert B_n \cdots B_1 \Vert^s.
\end{equation}
It is easy to see that the limit on the right hand side does exist, and the function $\lambda^{s}$ is convex in $s,$ and $\lambda^{s}/s$ is monotone non-decreasing. We now show that the same holds as a discrete series for the computable bound $\eta_{2k}/(2k)$, implying $k=1$ is optimal, in line with the choice in Theorem \ref{thm:Main}.

Towards showing the (mid-point) convexity of $\eta_{2k}$ we present an inequality of general interest, a Cauchy-Schwartz type comparison for tensor products. Among various versions available in the literature, the current proof is significantly shorter than the one found in \cite{gerencser2008lq}. The main ideas are built on \cite{pisier2003introduction}, referring to \cite{haagerup1985grothendieck}, but now interpreted in a simple way that is sufficient for our finite dimensional setting without needing to delve into the operator space theoretical context.

\begin{lemma}
  \label{lm:cs-tensor}
  Let us consider random matrices $X$ and $Y$.
  Then there is a constant $C>0$ depending on the dimensions such that
  \begin{equation*}
    \label{eq:cs-tensor}
    \| \mE (X\otimes Y) \| \le C \sqrt{\| \mE (X\otimes X) \|} \cdot \sqrt{\| \mE (Y\otimes Y) \|}.
  \end{equation*}
  For square matrices we further have
  \begin{equation*}
    \label{eq:cs-tensor-spectral}
    \rho( \mE (X\otimes Y) ) \le \sqrt{\rho(\mE(X\otimes X)) } \cdot \sqrt{\rho(\mE (Y\otimes Y)) }.
  \end{equation*}
\end{lemma}
\begin{proof}
We prove the claim with $C=1$ for a special choice of norms, then the claim follows by the equivalence of norms.
The construction is indirect at first, we handle an element $Z$ expressed as a mixture of tensor products by some auxiliary measure $\mu$ on some auxiliary set $S$, with proper choice of $X,Y$, i.e., assume $Z=\int_S X\otimes Y d\mu$.
Define the norm of such a mixture as
$$
\left\|\int_S X\otimes Y d\mu \right\|_* = \sup_{\|a_x\| = \|b_x\|=\atop \|a_y\| = \|b_y\| = 1}\left\{ \int_S a^\top_x X b_x b_y^\top Y^\top a_y d\mu \right\},
$$
where $a_x,b_x,a_y,b_y$ are vectors of appropriate dimensions. This is a general scheme for all tensor product spaces encountered.

First we need to check this is a norm indeed. It is defined: all big matrices can be expressed as mixtures of tensor products. It is well defined: observe that for any $a_x,b_x,a_y,b_y$ the integral depends only on the value, not the representation, then taking supremum keeps this property. Linearity is straightforward from the definition. For the triangle inequality we use the freedom of representation, w.l.o.g we may express an addition $Z=Z_1+Z_2$ as merging disjoint representations, i.e., take $S=S_1\cup S_2$ with $S_1\cap S_2 = \emptyset$ with $X,Y,\mu$ merging the action of $X_i,Y_i,\mu_i$. Then by construction 
$$
\int_S X\otimes Y d\mu = \int_{S_1} X_1\otimes Y_1 d\mu_{1} + \int_{S_2} X_2\otimes Y_2 d\mu_{2}.
$$
Similarly, towards computing the norm we have
\begin{align*}
  \int_S a^\top_x X b_x b_y^\top Y^\top a_y d\mu &= \int_{S_1} a^\top_x X_1 b_x b_y^\top Y_1^\top a_y d\mu_1\\
  &+ \int_{S_2} a^\top_x X_2 b_x b_y^\top Y_2^\top a_y d\mu_2.  
\end{align*}
Taking supremum in $a_x,b_x,a_y,b_y$ for the l.h.s.\ the norm appears there. These are not necessary the optimal parameter vectors for the r.h.s. expressions, immediately confirming the triangle inequality.

Let us turn to our main claim. Notice that the underlying probability space appears naturally for expressing the mixtures to handle. For fixed $a_x,b_x,a_y,b_y$, the claim is a standard ``Cauchy-Schwartz'' between the two random scalars $a^\top_x X b_x$ and $b_y^\top Y^\top a_y$, i.e.,
\begin{align*}
  \mE(a^\top_x X b_x b_y^\top Y^\top a_y) \le &\sqrt{\mE(a^\top_x X b_x b_x^\top X^\top a_x)}\\
  \cdot &\sqrt{\mE(a^\top_y Y b_y b_y^\top Y^\top a_y)}.  
\end{align*}
Optimizing in $a_x,b_x,a_y,b_y$ for the l.h.s. we get the inequality needed to conclude the first part on norms.
Considering the spectral radius, take $n$ independent copies of $X,Y$ and apply the claim for their products, noting 
$\mE((X_1\cdots X_n)\otimes (Y_1\cdots Y_n))=\mE (X\otimes Y)^n$
arriving at
$$
\| \mE (X\otimes Y)^n \|_* \le \sqrt{\| \mE (X\otimes X)^n \|_*} \cdot \sqrt{\| \mE (Y\otimes Y)^n \|_*}.
$$
Taking $n$th root and letting $n\to\infty$ 
the spectral radii appear as required.
\end{proof}

Now we can conveniently apply the above in our context:
\begin{lemma}
  \label{lm:eta_conv}
  In the setting of Assumption \ref{assumption:typicalsetup}, $\eta_{2k}$ is (mid-point) convex in $k$.
  Also, ${\eta_{2k}} /{(2k)}$ is non-decreasing.
\end{lemma}
\begin{proof}
  By applying Lemma \ref{lm:cs-tensor} with $X=B^{\otimes k-1}$ and $Y=B^{\otimes k+1}$ for $k\ge 1$ we directly get
  $$
  \exp(\eta_{2k}) \le \exp(\eta_{2k-2}/2) \exp(\eta_{2k+2}/2),
  $$
  showing the convexity of $\eta_{2k}$. To complete the sequence, extend to $\eta_0=0$. Lemma \ref{lm:cs-tensor} still provides convexity at $k=1$, using the identity matrix when necessary, this easily implies that $\eta_{2k}/(2k)$ must be non-decreasing for $k\ge 1$.
\end{proof}

\section{Numerical results}
\label{sec:simul}

The main question is the sharpness of the upper bound on the exponential convergence rate obtained.
We do not launch the processes from a single initial $x_0$, but rather from $(I-J)$:
\begin{equation}
  \label{eq:simrate}
  \frac 1 n \log \left(\frac 1{\sqrt{p}} {\| \diag(w_n)^{-1}M_n(I-J)\|_F}\right),
\end{equation}
bounding the empirical rate of convergence of the process for the range of any starting vector spanned by the columns of $I-J$: those with 0 average.
We take $n=100000$, which is generous in view of the size of the network and the communication pattern to follow. The Julia computing platform is used to carry out the simulations \cite{Julia2017}, \cite{JuliaGraphs2021}.

For the underlying network, we consider a model based on Random Geometric Graphs (RGG) \cite{penrose2003random}
with a simple perturbation where a dependence on the positions of the agents can be introduced.
We interpolate between a grid and uniform random placement.
To be more precise, let $c\in [0,1]$ be an interpolation parameter together with a reference number of nodes, $p_0$. For each node $i$, two preliminary positions are assigned: $z^i_g$ a unique point on the $\sqrt{p_0}\times\sqrt{p_0}$ square grid fitted into $[0,1]^2$ and $z^i_r$, a uniform random position in $[0,1]^2$. The final position is then declared as $z^i=cz^i_r+(1-c)z^i_g$.

We still need to define the graph on the points obtained on the unit square. We still want to stay with the concept of connecting those that are close.

As the structure of positions are changing, the clear connectivity thresholds for RGGs \cite{gupta1999RGG} does not apply anymore. Instead to get a graph with balanced density, we optimize the threshold for the distance of two nodes getting connected
so that the largest connected component contains $\approx 90\%$ of the nodes. Then this giant component is kept for further work also determining the final dimension $p$.
In Figure \ref{fig:gridrrg_graph} we see two examples for $c=0.1$ and $c=0.8$ for $p_0=64$ initial points, which will be the default size parameter for our simulations. This leads to a typical dimension of $p\approx 58$.
\begin{figure}[h]
  \centering
  \subfloat[ ]{
    \includegraphics[width=0.18\textwidth]{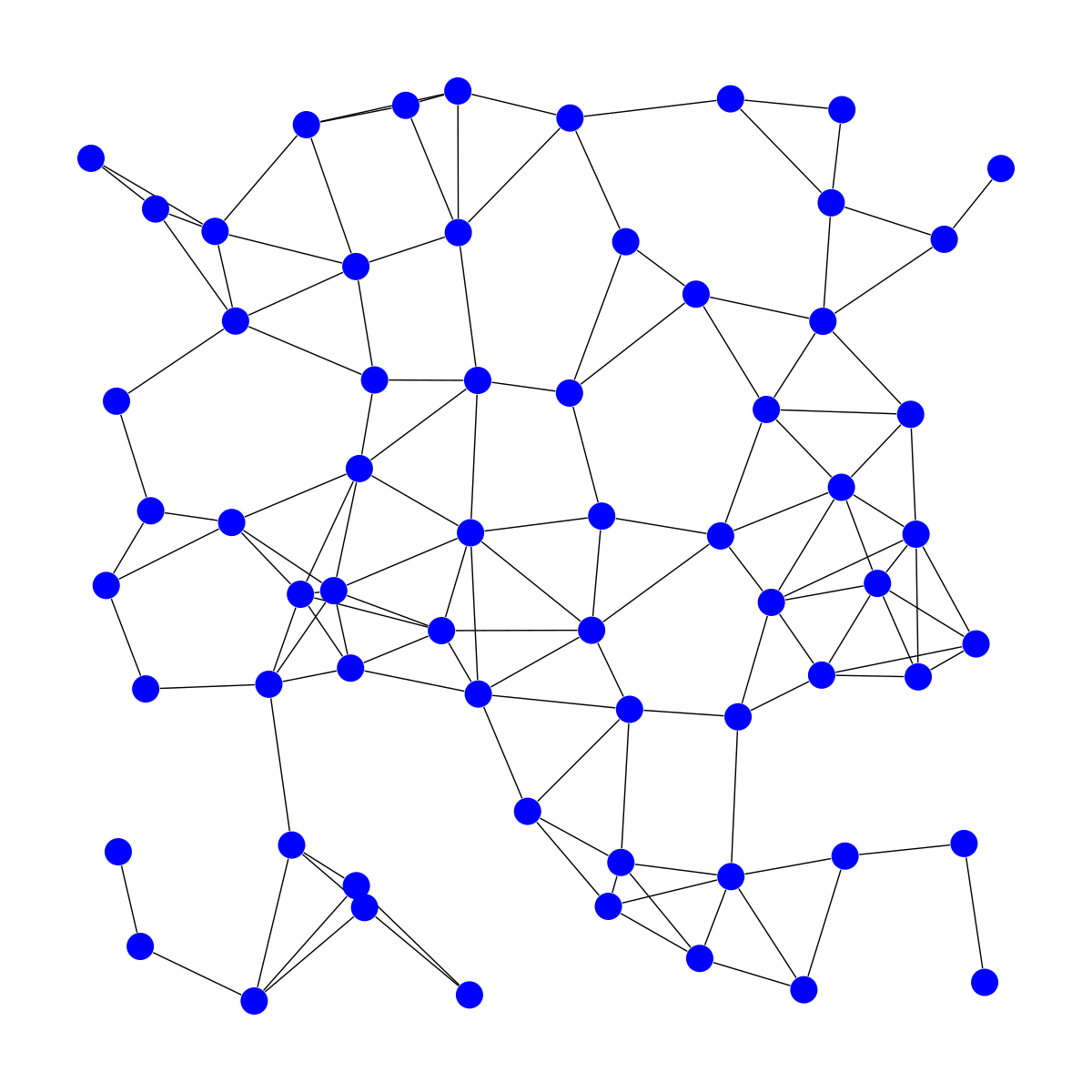}
  }
  \subfloat[ ]{
    \includegraphics[width=0.18\textwidth]{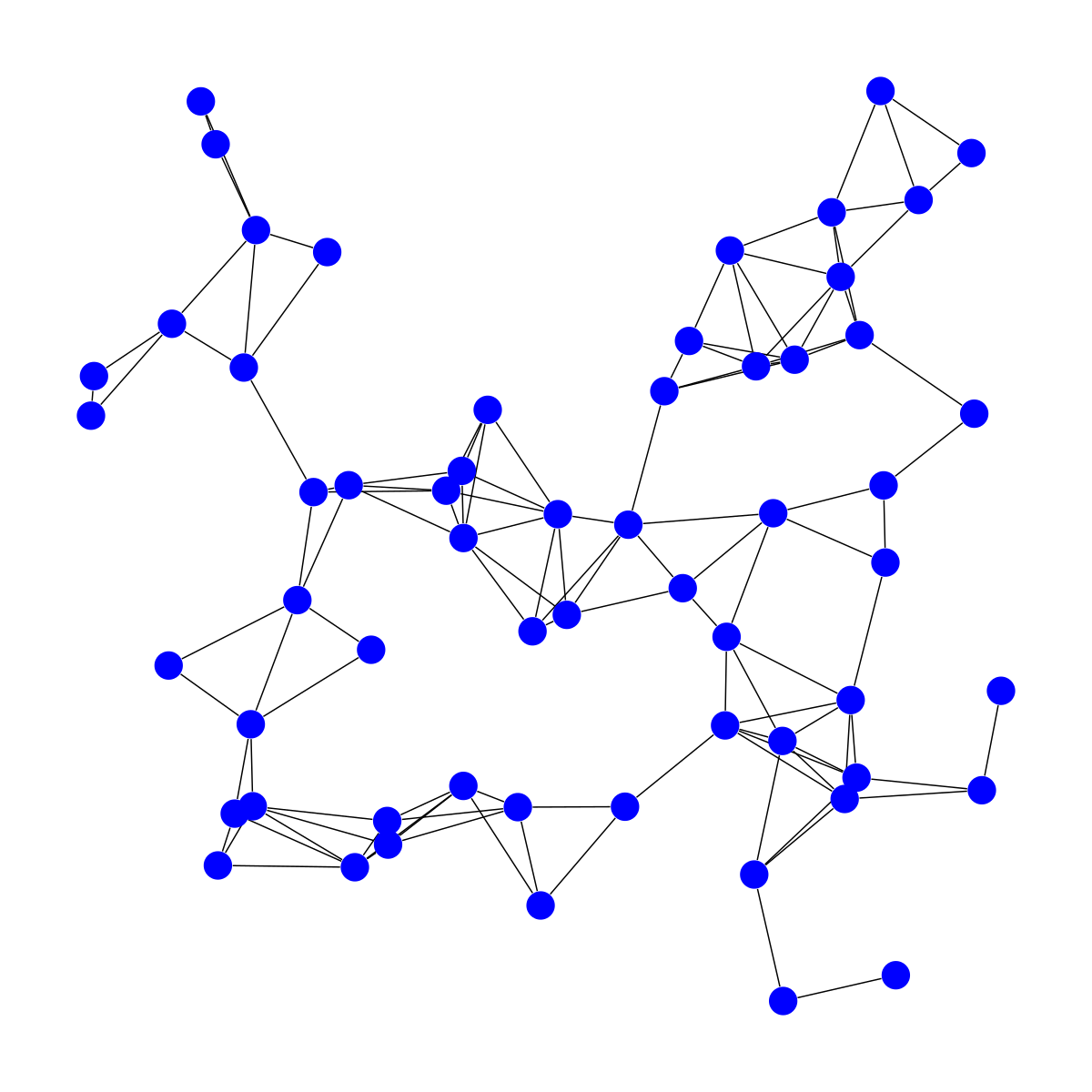}
  }
  \caption{Example random graphs interpolating between square grid and uniform random node placement, for (a) $c=0.1$ and (b) $c=0.8$ coefficient for the random component.}
  \label{fig:gridrrg_graph}
\end{figure}

The reference dynamics is asynchronous directed gossip: every step a uniformly chosen node communicates towards a single uniformly chosen neighbor, sending $\frac 1 2$ fraction of its value and weight. Figure \ref{fig:grid_gossip} presents the empirical rate according to \eqref{eq:simrate} together with $\eta_2/2$ for 500 simulations for various $c$, we see the two move together despite the wild randomness of the graph instances. We also see the difference of the two, showing that it is reliable estimate even point-wise.
The numerical stability is demonstrated by a single instance out of the 500 when there is a positive difference of $\approx 2\cdot 10^{-6}$.
\begin{figure}[h]
  \centering
    \includegraphics[width=0.50\textwidth]{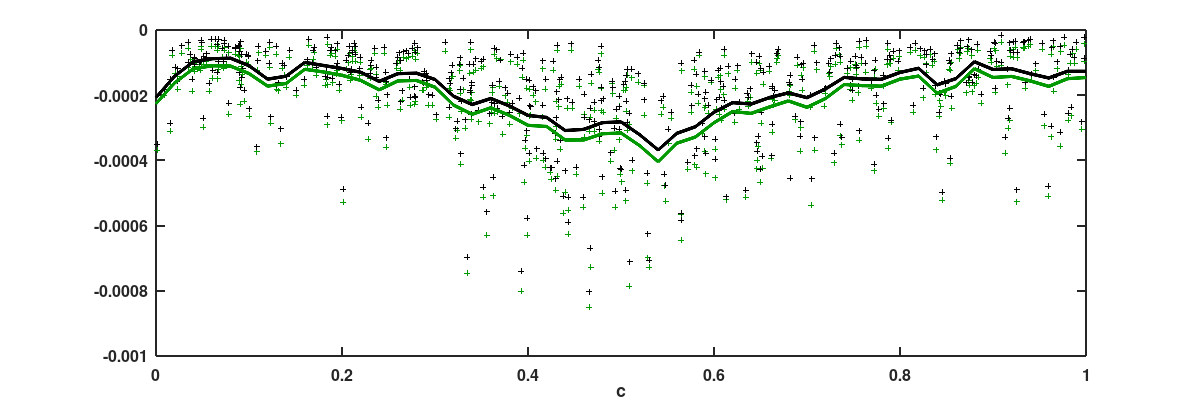}
    \includegraphics[width=0.50\textwidth]{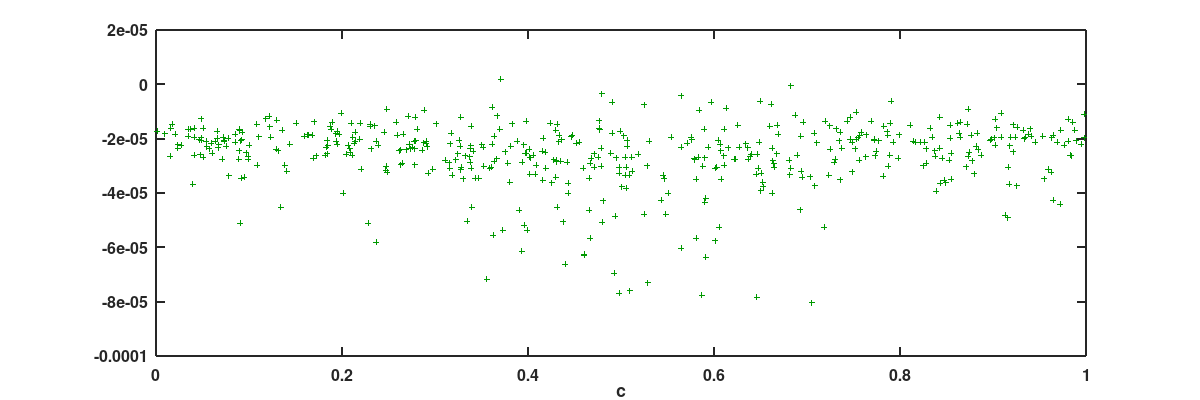}
    \caption{Empirical (green) rates versus bounding $\eta_2/2$ (black) and their moving averages, below the difference of the two.}
  \label{fig:grid_gossip}
\end{figure}

We compare the reference gossip with two modified strategies. First, we have two-way randomized gossip: the transmitting node selects two receivers and uniformly randomly splits the total fraction of $\frac 1 2$ to be sent between them. Note that the conditions of Theorem \ref{thm:Main}, Assumption \ref{assumption:typicalsetup} still holds, here $\cA$ is a union of segments in the space of non-negative matrices. Second, the reference gossip is modified to send only a fraction of $\frac 1 4$ to a single recipient, but we allow twice as many steps to take place, we will name this slowed gossip for convenience. Figure \ref{fig:grid_comparison} presents the comparison of rates of the modified strategies.
For the empirical rates we observe no consistent ordering of the three strategy, however, the difference is an order of magnitude smaller than the variance caused by by the graph variability, see the range in Figure \ref{fig:grid_gossip}. $\eta_2/2$ provides two-way gossip a stronger bound than for the reference process, even more for the slowed gossip.
\begin{figure}[h]
  \centering
    \includegraphics[width=0.50\textwidth]{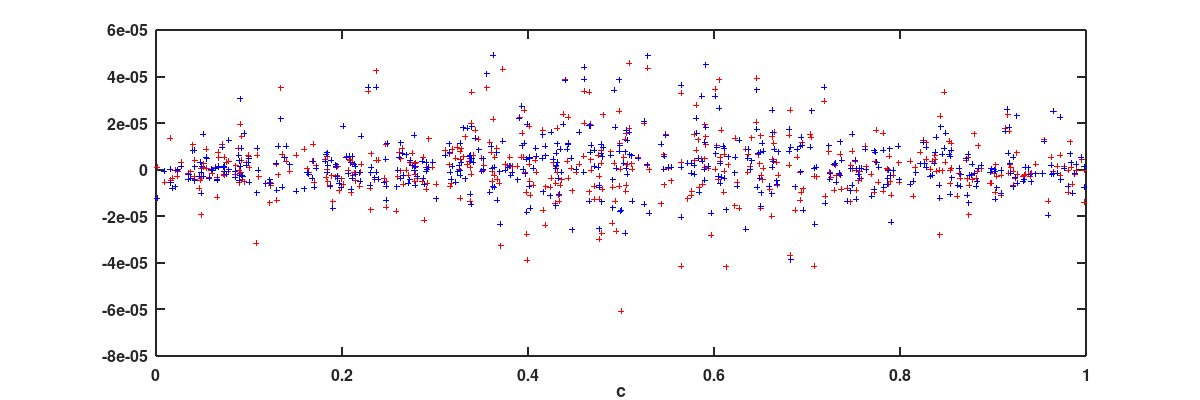}
    \includegraphics[width=0.50\textwidth]{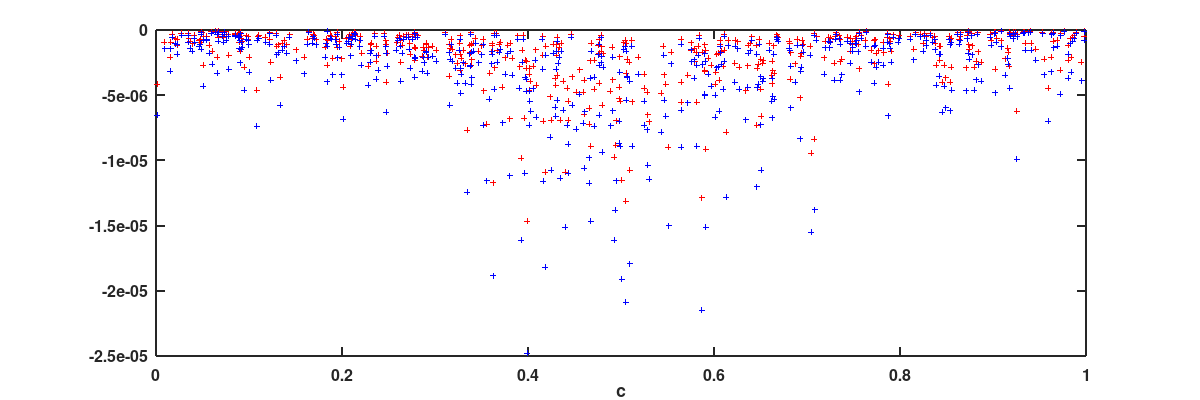}
    \caption{Empirical rates (above) and $\eta_2/2$ (below) compared to that of the reference gossip on structured RGGs, for two-way randomized gossip (red) and slowed gossip (blue).}
  \label{fig:grid_comparison}
\end{figure}

Another natural question to ask is the dependence of the rate on the connection structure.
We consider the following Erdős-Rényi process inspired model to study this phenomenon: starting with a cycle on $p=50$ nodes we add random edges uniformly one by one, up to $500$ (an average extra degree of $20$), and at each step, we evaluate both the empirical rates and $\eta_2/2$ of the three process variants. Figure \ref{fig:cycle_all} shows the aggregated picture. It is apparent that initially in the sparse region an extra edge is much more game-changing than later on. Also, in the late phase there is a clear ordering of efficiency of slowed gossip being the best, followed by two-way gossip, then the reference. Note however, this ordering is not fully present in the earlier sparse, less interconnected phases.

\begin{figure}[h]
  \centering
  \includegraphics[width=0.5\textwidth]{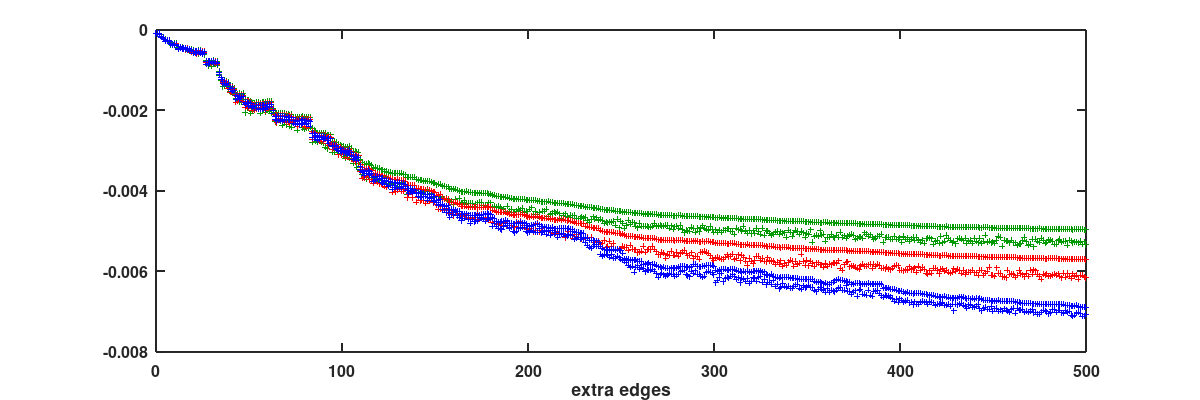}
  \caption{Empirical rates and $\eta_2/2$ for reference gossip (green) two-way randomized gossip (red) and slowed gossip (blue) on the cycle with random edges.}
  \label{fig:cycle_all}
\end{figure}

Let us zoom in once again for this connection structure to compare the three process variants. We compare separately the empirical rates and the computed $\eta_2/2$ for both process variants, normalized against the reference gossip in Figure \ref{fig:cycle_compare}. We observe that in terms of the empirical rate, initially there is a non-trivial race among the three strategies, then between the two-way and the slowed gossip, before the final order is settled for dense graphs. Interestingly, for the bounding $\eta_2/2$ the ordering is robust. 
\begin{figure}[h]
  \centering
  \includegraphics[width=0.5\textwidth]{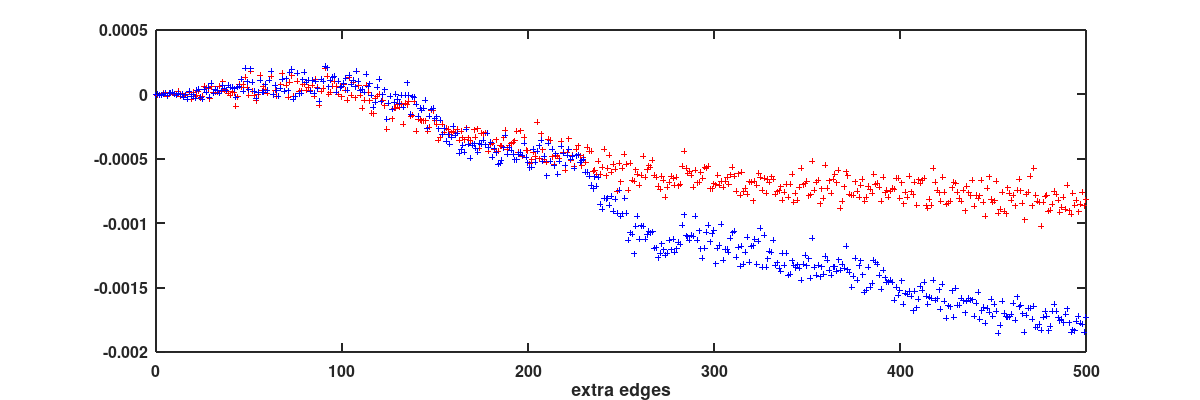}
  \includegraphics[width=0.5\textwidth]{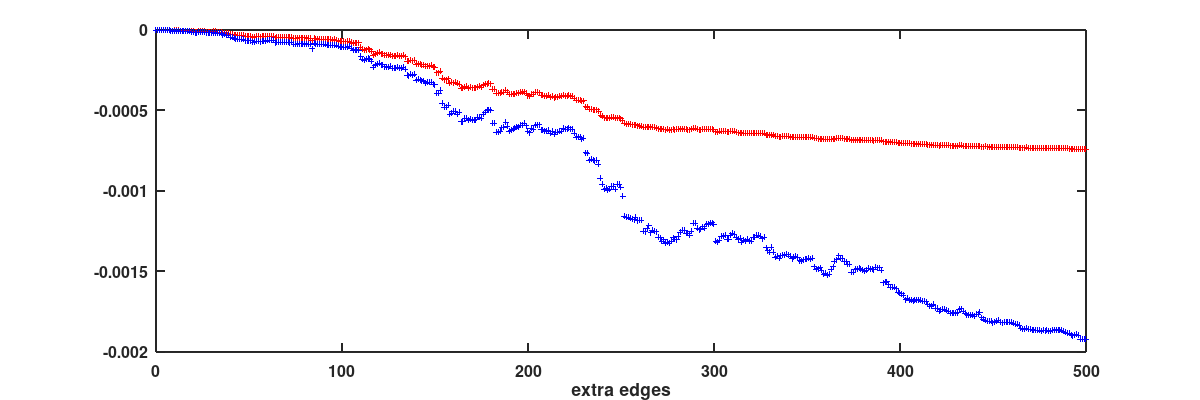}
  \caption{Empirical rates (above) and $\eta_2/2$ (below) compared to that of the reference gossip on the cycle with random edges, for two-way randomized gossip (red) and slowed gossip (blue).}
  \label{fig:cycle_compare}
\end{figure}

\section{Conclusion}
\label{sec:concl}

We have proven upper bounds for the almost sure exponential convergence rate of i.i.d.\ ratio consensus algorithms
inspired by the approach of \cite{iutzeler2013analysis} and by the analysis in \cite{gerencsr2019tight}.
The quantity $\eta_{2k} = \log\rho\left(\mE(A_{1}^{\otimes 2k})(I-J)^{\otimes 2k}\right)$ is indeed accessible, as it is based on the spectral description of a finite transformation of the matrix distribution describing the updates. 
We have shown that $\eta_{2k}/(2k)$ is non-decreasing, thus for bounding the convergence rate it is optimal to keep $k=1$.
However, our general results can be applied to provide upper bounds on the convergence rate of higher moments for linear consensus.

Through numerical examples we have observed that the bounds tend to capture well the magnitude of the rate, with an error of lower order. Also, for sparse networks few additional edges can improve efficiency significantly.

\section{Appendix}

\begin{proof}[of Lemma \ref{lem:SEQ_PRIM_IID}] 
	Define $\gamma:[0,\infty) \to \{0,1\}$ simply as the indicator of being positive. This naturally extends to matrices element-wise.
	
	The primitivity of $\mathcal A$ is characterized by the primitivity of $\gamma(\mathcal A)$ as only positivity is needed through the process, without focus on the actual value, and we are working with non-negative matrices.
	
	To investigate the support of the push-forward measure $\gamma_* \mu$, choose any $A\in\mathcal A$ and its projection $A^0=\gamma(A)$. Let $\eps = \min_{i,j}\{A^{i,j}\mid A^{i,j}>0\}$. For the small ball $B(A,\eps/2)$ we must have $\mu(B(A,\eps/2))>0$ as $A$ is in the support by assumption. Observe there are two type of matrices in the ball: $\cB_1$ with some with negative elements, not playing a role, and $\cB_2$ with non-negative matrices where positive elements appear at least where $A$ has them.
	Therefore $\mu(\cB_1)=0,\mu(\cB_2)>0$, and for any matrix of $B\in\cB_2$ we have $\gamma(B)\ge A^0$. This means $\gamma(\cB_2)$ maps to a set with positive $\gamma_* \mu$ probability, with all matrices bounded below by $A^0$.
	Consequently, the support of $\gamma(\mathcal{A})$ is majorated by the support of $\gamma_* \mu$, in the sense that for any matrix $A^0\in \gamma(\mathcal{A})$ there exists $B^0\in \supp(\gamma_* \mu)$ such that $A^0\le B^0$. Note that the two supports are not necessary equal.
	
	Consider now any sequence $A_{k_l}\cdots A_{k_1}>0$ with $A_{k_i}\in\mathcal{A}$ presenting the primitivity of $\mathcal{A}$. We know it is equivalent to $\gamma(A_{k_l})\cdots\gamma(A_{k_1})>0$. By the previous argument, we have matrices $\gamma(A_{k_i})\le B_{i}^0\in \supp(\gamma_* \mu)$, thus $B_{l}^0\cdots B_{1}^0>0$. Being in the support on a discrete space means $B_{i}^0$ has positive probability to appear, similarly for the chosen product at any $l$ consecutive steps, thus it will eventually occur (as we have an i.i.d.\ process), confirming sequential primitivity. In the meantime we rely on the matrices being allowable so that it is sufficient to find a positive product at an arbitrary starting time.
\end{proof}

\begin{proof}[of Lemma \ref{lem:EA_1_Irreducible}]
	We prove by contradiction, let us assume $\mE A_{1}$ is reducible.
	Without the loss of generality we can assume that $\mE A_{1} $ has the block structure $\big(\begin{smallmatrix}
	B & C\\
	0 & D
	\end{smallmatrix}\big)$
	with square blocks in the diagonal. Knowing that $A_1$ is non-negative, this would force $A_1$ to have the same structure a.s., then by stationarity for all $A_n$, and then for their products of any length, contradicting sequential primitivity.
	Thus $\mE A_{1} $ indeed must be irreducible. 
\end{proof}

\begin{proof}[of Lemma \ref{lem:w_subexp1}]
  The claim is a direct consequence of a Lemma 44 in \cite{gerencsr2019tight}, the validity of its conditions has been verified in Section \ref{sec:tightbounds}.
  The cited lemma states that the product $M_{n} = A_{n}A_{n-1}\cdots A_1$ is asymptotically rank-1, more specifically for any fixed pair of rows $i,j$ and any column $k$ the ratio ${M_{n}^{ik}}/{M_{n}^{jk}}$ is sub-exponential.
 
\medno 
As $w_n=M_{n}\b1$, each $w_n^i/w_n^j$ is easily seen to be a convex combination of the corresponding quotients ${M_{n}^{ik}}/{M_{n}^{jk}}, \, k = 1, \ldots, p.$ Thus  $w_n^i/w_n^j$ is also sub-exponential. Recall that $A_m$ is column stochastic for all $m$, and hence $M_{n}$ is also  column stochastic for all $n$. Thus we have $\m1^\top w_n = p$ for all $n$. Summation of $w_n^i/w_n^j$ through $i,$ with $j$  fixed, yields $p/w_n^j.$ Since each term is sub-exponential, it follows that $p/w_n^j,$ and its maximum over $j,$ is sub-exponential as well.
\end{proof}

\bibliographystyle{ieeetr}
\bibliography{pushsum,current,ringmixing}

\end{document}